\documentclass[11pt]{amsart}
\usepackage{multirow,bigdelim}
\usepackage{amsfonts}
\usepackage{dsfont}
\usepackage{amsmath}
\usepackage{mathrsfs}
\usepackage{todonotes}
\usepackage{hyperref}
\usepackage{multimedia}
\usepackage{graphicx}
\usepackage{indentfirst}
\usepackage{bm}
\usepackage{amsthm}
\usepackage{mathrsfs}
\usepackage{latexsym}
\usepackage{amsmath}
\usepackage{amssymb}
\usepackage{microtype}
\usepackage{relsize}
\usepackage{hyperref}
\newcommand{\A}{\mathcal{A}}
\newcommand{\E}{\mathbb{E}}

\def\A{\mathcal A}
\def\B{\mathcal B}

\def\K{\mathcal K}

\def\B{\mathcal B}

\def\H{\mathcal H}
\def\K{\mathcal K}
\def\L{\mathcal L}
\def\M{\mathcal M}

\def\CCC{\mathbb C}

\def\NNN{\mathbb N}

\def\amslatex{$\mathcal{A}\kern-.1667em\lower.5ex\hbox{$M$}\kern-.125em\mathcal{S}$-\LaTeX}

\DeclareSymbolFont{SY}{U}{psy}{m}{n}
\DeclareMathSymbol{\emptyset}{\mathord}{SY}{'306}

\theoremstyle{plain}

\newtheorem{thm}{Theorem}[section]
\newtheorem{cor}[thm]{Corollary}
\newtheorem{lem}[thm]{Lemma}

\theoremstyle{definition}

\numberwithin{equation}{section}
\begin{document}
%%%%%%%%%%%%%%%%%%%%%%%%%%%%%%%%%%%%%%%%%%%%%%%%%%%%%%%%%%%%%%%%%%%%%%%%
\title{strong sums of projections in type ${\rm II}$ factors}
\author{ Xinyan Cao, Junsheng Fang and Zhaolin Yao}
%    Current address
\curraddr[X. Cao and J. Fang]{School of Mathematical Sciences, Hebei Normal University,
Shijiazhuang, Hebei 050016, China}
\curraddr[Z. Yao]{Postdoctoral Research Station of Mathematics, Hebei Normal University, Shijiazhuang, Hebei 050016, China}
\email[X.Cao]{caoxinyan123@qq.com}
\email[J.Fang]{jfang@hebtu.edu.cn}
\email[Z.Yao]{zyao@hebtu.edu.cn}
%\curraddr{Department of Mathematics, Hebei Normal University,
%Shijiazhuang, Hebei 050016, China}

\thanks{The authors were supported  by National Natural
Science Foundation of China (Grant No.12071109) and a start up funding from Hebei Normal Univeristy.}

\begin{abstract}
Let $\M$ be a type ${\rm II}$ factor and let $\tau$ be the faithful positive semifinite normal trace, unique up to scalar multiples in the type ${\rm II}_\infty$ case
and normalized by $\tau(I)=1$ in the type ${\rm II}_1$ case. Given $A\in \M^+$, we denote by $A_+=(A-I)\chi_A(1,\|A\|]$ the excess part of $A$ and by $A_-=(I-A)\chi_A(0,1)$
the defect part of $A$. In~\cite{KNZ}, V. Kaftal, P. Ng and S. Zhang provided necessary and sufficient conditions for a positive operator to be the sum of a finite or infinite collection of projections (not necessarily mutually orthogonal) in type ${\rm I}$ and type ${\rm III}$ factors. For type ${\rm II}$ factors, V. Kaftal, P. Ng and S. Zhang proved that $\tau(A_+)\geq \tau(A_-)$ is a necessary condition for an operator $A\in \M^+$ which can be written as the sum of a finite or infinite collection of projections and also sufficient if the operator is ``diagonalizable". In this paper, we prove that if $A\in \M^+$ and $\tau(A_+)\geq \tau(A_-)$, then $A$ can be written as the sum of a finite or infinite collection of projections. This result answers affirmatively  Question 5.4 of~\cite{KNZ}.
\end{abstract}

\maketitle
\section{Introduction}
Which positive bounded linear operators on a separable Hilbert space can be written as sums of projections? If the underlying space is finite dimensional, then a characterization of such operators was obtained by Fillmore~\cite{Fil}: a finite-dimensional operator is the sum of projections if and only if it is positive, it has an integer trace and the trace is greater than or equal to its rank.

For infinite sums with convergence in the strong operator topology, this question arose naturally from work on ellipsoidal tight frames by Dykema, Freeman, Kornelson, Larson, Ordower, and Weber in~\cite{DFK}. They proved that a sufficient condition for a positive bounded linear operator $A\in \B(\H)^+$ to be the sum of projections is that its essential norm $\|A\|_e$ is larger than one. The same question can be asked relative to a von Neumann algebra $\M$. This question was first asked by Kaftal, Ng and Zhang in~\cite{KNZ}. In~\cite{KNZ}, Kaftal, Ng and Zhang developed beautiful techniques and addressed the question completely for positive operators in type ${\rm I}$ and type ${\rm III}$ factors. They proved the following results. Let $\M$ be of type ${\rm I}$ factor. Then $A$ is a strong sum of projections if and only if either ${\rm Tr}(A_+)=\infty$ or ${\rm Tr}(A_-)\leq {\rm Tr}(A_+)<\infty$ and ${\rm Tr}(A_+)-{\rm Tr}(A_-)\in \NNN\cup\{0\}$. Let $\M$ be of type ${\rm III}$ factor. Then $A$ is a strong sum of projections if and only if either $\|A\|>1$ or $A$ is a projection. For type ${\rm II}$ factor, they proved that if $A$ is diagonalizable, then $A$ is a strong sum of projections if and only if  $\tau(A_+)\geq \tau(A_-)$. The condition  $\tau(A_+)\geq \tau(A_-)$ is necessary even when $A$ is not diagonalizable. Recall that a positive operator $A$ is said to be diagonalizable if $A=\sum \gamma_j E_j$ for some $\gamma_j>0$ and mutually orthogonal projections $\{E_j\}$ in $\M$.
The main result of this article is the following theorem.

\begin{thm}
Let $\M$ be a type ${\rm II}$ factor and let $\tau$ be the faithful positive semifinite normal trace, unique up to scalar multiples in the type ${\rm II}_\infty$ case
and normalized by $\tau(I)=1$ in the type ${\rm II}_1$ case. Given $A\in \M^+$, we denote by $A_+=(A-I)\chi_A(1,\|A\|]$ the excess part of A and by $A_-=(I-A)\chi_A(0,1)$
the defect part of $A$. If $A\in \M^+$ and $\tau(A_+)\geq \tau(A_-)$, then $A$ can be written as the sum of a finite or infinite collection of projections.
\end{thm}

This result answers affirmatively Question 5.4 of~\cite{KNZ}. To prove this theorem, we develop some new techniques in this paper. On the other hand, several results, e.g., Proposition 3.1, Theorem 5.2 and Theorem 6.6, in~\cite{KNZ}  play important roles in the proof.

The organization of this paper is as follows. In section 2 below, we prove the following result. Let $\M$ be a semifinite factor with a faithful normal trace $\tau$ and let $X=X^*\in \M\cap L^1(\M,\tau)$ such that $\tau(X)=0$. Then there is a family of nonzero projections $\{E_\alpha\}$ in $\M$ such that $\sum_\alpha E_\alpha=I$ and $E_\alpha XE_\alpha=0$ for all $\alpha$. This result is inspired by~\cite{C-W} and plays a key role throughout the paper. In section 3, we prove that let $\M$ be a type ${\rm II}_1$ factor and let $\tau$ be the faithful positive normal trace normalized by $\tau(I)=1$. If $A\in \M^+$ and $\tau(A_+)\geq \tau(A_-)$, then $A$ can be written as the sum of a finite or infinite collection of projections. In section 4, we consider the similar question in type ${\rm II}_\infty$ factors. We prove that let $\M$ be a type ${\rm II}_\infty$ factor and let $\tau$ be a faithful positive normal tracial weight on $\M$. If $A\in \M^+$ and $\tau(A_+)\geq \tau(A_-)$, then $A$ can be written as the sum of a finite or infinite collection of projections. In section 5, we give a new proof of the main theorem of~\cite{KNZ} for type ${\rm I}$ factors.

We refer to~\cite{C-W,G-P,KNZ1,KNZ2,KNZ3,KNZ4,KRS,KRS2,Mar,Nak} for interesting results on the related topics. We refer to~\cite{K-R,S-Z,Tak} for the basic knowledge of von Neumann algebras.

\section{Isotropic subspaces of operators in semifinite factors}

Let $T$ be a bounded linear operator on a Hilbert space $\H$. A (closed) subspace $\K$ of $\H$ is \emph{isotropic} for $T$ if $\langle T\xi,\xi\rangle=0$ for all $\xi\in \K$. This is the same as saying that $P_\K T P_\K=0$.

\begin{lem}\label{L:lemma 2.1}
Let $\M$ be a semifinite factor with a faithful normal trace $\tau$ and let $X=X^*\in \M\cap L^1(\M,\tau)$ such that $\tau(X)=0$. Then there is a nonzero projection $E\in \M$ such that $EXE=0$.
\end{lem}

To prove Lemma~\ref{L:lemma 2.1}, we need the following two lemmas.

\begin{lem}\label{L:Lemma 2.2}
Let $\M$ be a von Neumann algebra and let $X\in \M$. Then there is a nonzero projection $E\in \M$ such that $EXE=0$ if and only if there is a nonzero element $Y \in \M$ such that $YXY^*=0$.
\end{lem}
\begin{proof}
Suppose $Y\neq 0$ satisfies $YXY^*=0$. By the polar decomposition theorem, there is a nonzero positive operator $H=(Y^*Y)^{1/2}\in\M$ and a partial isometry $V$ in $\M$ such that $Y=VH$ and $V^*VH=H=HV^*V$. Then $YXY^*=0$ implies that
\[
VHXHV^*=0
\]
and therefore,
\[
HXH=V^*VHXHV^*V=0.
\]
 Let $F=\chi_{H}(\{0\})$ and $E=I-F$. Then $E\neq 0$ and
\[
H=\begin{pmatrix}
H_{1}&0\\
0&0
\end{pmatrix}
\]
with respect to the decomposition $I=E+F$, where $H_{1}=EHE$ is an injective positive map with dense range.
Write
\[
X=\begin{pmatrix}
X_{11}&X_{12}\\
X_{21}&X_{22}
\end{pmatrix}
\]
with respect to the decomposition $I=E+F$. Then
\[
0=HXH=\begin{pmatrix}
H_{1}&0\\
0&0
\end{pmatrix}
\begin{pmatrix}
X_{11}&X_{12}\\
X_{21}&X_{22}
\end{pmatrix}
\begin{pmatrix}
H_{1}&0\\
0&0
\end{pmatrix}=\begin{pmatrix}
H_{1}X_{11}H_1&0\\
0&0
\end{pmatrix}
\]
 implies that $H_1X_{11}H_1=0$. Since $H_1$ is injective with dense range, $X_{11}=0$. This implies that $EXE=0$.

\end{proof}

\begin{lem}\label{L:lemma 2.3}
Let $(\M,\tau)$ be a  factor and let $U=U^*\in \M$ be a unitary operator such that $U\neq I,-I$. Then there is a nonzero projection $E$ such that $EUE=0$.
\end{lem}
\begin{proof}
Let $F_1=\chi_{U}(\{1\})$ and $F_2=\chi_{U}(\{-1\})$. Then $F_1,F_2\neq 0$, $F_1+F_2=I$ and $U=F_1-F_2$. Without loss of generality, we assume $F_1\prec F_2$. Then there is a subprojection $F$ of $F_2$ such that $F\sim F_1$. Let $V=F_1-F$. Note that $(F+F_1)\M (F+F_1)\cong M_2(\CCC)\otimes F\M F$ and a simple computation shows that
\[
\begin{pmatrix}
\frac{1}{2}&\frac{1}{2}\\
\frac{1}{2}&\frac{1}{2}
\end{pmatrix}
\begin{pmatrix}
1&0\\
0&-1
\end{pmatrix}
\begin{pmatrix}
\frac{1}{2}&\frac{1}{2}\\
\frac{1}{2}&\frac{1}{2}
\end{pmatrix}
=0.
\]
So there is a nonzero subprojection $E$ of $F_1+F$ such that $EVE=0$. Now
\[
EUE=E(F_1+F)U(F_1+F)E=EVE=0.
\]
\end{proof}

\noindent \emph{Proof of Lemma~\ref{L:lemma 2.1}.}
If $\chi_{X}(\{0\})\neq 0$, then let $E=\chi_{X}(\{0\})$. We have $EXE=0$. So we may assume that $\chi_{X}(\{0\})=0$, i.e., $X$ is injective with dense range. By the polar decomposition theorem, there is a unitary operator $U=U^*$ and an injective positive operator $H$ such that $X=UH$. Since $\tau(X)=0$, $U\neq I,-I$.

Assume that $\M$ is a finite factor. By Lemma~\ref{L:lemma 2.3}, there is a nonzero projection $F$ such that $FUF=0$. Now $FH^{-1/2}UHH^{-1/2}F=0$ and $FH^{-1/2}\neq 0$. Let $Z=FH^{-1/2}$. Then $Z$ is a closed densely defined operator affiliated with $\M$ and $Z^*=H^{-1/2}F$. Let $Y=(I+ZZ^*)^{-1/2}Z$. Then $Y$ is a bounded operator in $\M$ and $Y\neq 0$. Note that $Y^*=Z^*(I+ZZ^*)^{-1/2}$ and $YXY^*=0$. By Lemma~\ref{L:Lemma 2.2}, there is a nonzero projection $E\in \M$ such that $EXE=0$.

Assume that $\M=\B(\H)$ and $\H$ is infinite dimensional. Since $X$ is a trace class operator, $X$ is diagonalizable. Since ${\rm Tr}(X)=0$, we may assume that $X=\sum_n\lambda_n e_n\otimes e_n$ and $\lambda_1>0,\lambda_2<0$. Let
\[
Y=\begin{pmatrix}
\frac{1}{2}&\frac{1}{2}\\
\frac{1}{2}&\frac{1}{2}
\end{pmatrix}
\begin{pmatrix}
\lambda_1^{-1/2}&0\\
0&(-\lambda_2)^{-1/2}
\end{pmatrix}\in \L(\CCC e_1\oplus \CCC e_2).
\]
Then a simple calculation shows that $Y(\lambda_1 e_1\otimes e_1+\lambda_2 e_2\otimes e_2)Y^*=0$. Now $(Y\oplus 0)X(Y^*\oplus 0)=0$. By Lemma~\ref{L:Lemma 2.2}, there is a nonzero projection $E$ such that $EXE=0$.

Assume that $\M$ is type ${\rm II}_\infty$.  Since $X\in \M\cap L^1(\M,\tau)$, $\chi_{U}(\{1\})H,-\chi_{U}(\{-1\})H\in \M\cap L^1(M,\tau)$. Thus
 there are nonzero finite projections $E_1\leq \chi_{U}(\{1\})$ and $E_2\leq \chi_{U}(\{-1\})$ such that $E_1X=XE_1$ and $E_2X=XE_2$.
 Let $\tilde{X}=(E_1+E_2)U(E_1+E_2)H=\tilde{U}\tilde{H}$. Then $(E_1+E_2)\M (E_1+E_2)$ is a type ${\rm II}_1$ factor and $\tilde{X}\in (E_1+E_2)\M (E_1+E_2)$.
 By Lemma~\ref{L:lemma 2.3}, there is a nonzero projection $F$ such that $F\tilde{U}F=0$. Now
 \[
 F\tilde{H}^{-1/2}\tilde{U}\tilde{H}\tilde{H}^{-1/2}F=0.
 \]
 Let $Z=F\tilde{H}^{-1/2}$. Then $Z$ is a closed densely defined operator affiliated with $(E_1+E_2)\M (E_1+E_2)$ and $Z^*=\tilde{H}^{-1/2}F$. Let $Y=(I+ZZ^*)^{-1/2}Z$. Then $Y$ is a bounded operator in $(E_1+E_2)\M (E_1+E_2)$ and $Y\neq 0$. Since $Y^*=Z^*(I+ZZ^*)^{-1/2}$,
 $Y((E_1+E_2)UH)Y^*=Y\tilde{X}Y^*=0$. Now $(Y\oplus 0)X(Y^*\oplus 0)=0$. By Lemma~\ref{L:Lemma 2.2}, there is a nonzero projection $E$ such that $EXE=0$.

\begin{thm}\label{T:theorem 2.4}
Let $\M$ be a semifinite factor with a faithful normal trace $\tau$ and let $X=X^*\in \M\cap L^1(\M,\tau)$ such that $\tau(X)=0$. Then there is a family of nonzero projections $\{E_\alpha\}$ in $\M$ such that $\sum_\alpha E_\alpha=I$ and $E_\alpha XE_\alpha=0$ for all $\alpha$.
\end{thm}
\begin{proof}
Let $\E=\{\{E_\alpha\}\}$ be the family of mutually orthogonal nonzero projections such that $E_\alpha XE_\alpha=0$ for all $\alpha$. By Lemma~\ref{L:lemma 2.1}, $\E$ is not empty. Define $\{E_\alpha\}\prec \{F_\beta\}$ if $\{E_\alpha\}\subseteq \{F_\beta\}$. Clearly, $\prec$ is a partial order and if $\{\{E_\alpha\}\}$ is a totally ordered set, then the union is a maximal element of the totally ordered set. By Zorn's lemma, there is a maximal element $\{E_\alpha\}$ in $\E$. Claim $\sum_\alpha E_\alpha=I$. Otherwise, let $F=I-\sum_{\alpha}E_\alpha$.
Then $FXF\in \M_F=F\M F$. If $\tau(F)<\infty$, then
\[
\tau_{\M_F}(FXF)=\frac{\tau(FXF)}{\tau(F)}=\frac{\tau(XF)}{\tau(F)}=\frac{\tau(X)-\tau(X\sum_\alpha E_\alpha)}{\tau(F)}=\frac{\tau(X)-\sum_\alpha\tau(E_\alpha XE_\alpha)}{\tau(F)}=0.
\]
By Lemma~\ref{L:lemma 2.1}, there is a nonzero projection $E\leq F$ such that $EFXFE=0$. Thus $EXE=0$ and $\{E_\alpha\}\cup \{E\}\in \E$. This is a contradiction.
If $\tau(F)=\infty$, then
\[
\tau(FXF)=\tau(XF)=\tau(X)-\tau(X\sum_\alpha E_\alpha)=\tau(X)-\sum_\alpha\tau(E_\alpha X E_\alpha)=0.
\]
By Lemma~\ref{L:lemma 2.1}, there is a nonzero projection $E\leq F$ such that $EFXFE=0$. Thus $EXE=0$ and $\{E_\alpha\}\cup \{E\}\in \E$. This is a contradiction.
\end{proof}

\begin{cor}\label{C:Corollary 2.5}
Let $\M$ be a countably decomposable semifinite factor with a faithful normal trace $\tau$ and let $X=X^*\in \M\cap L^1(\M,\tau)$ such that $\tau(X)=0$, e.g., $\M$ be type ${\rm II}_1$. Then there is a family of countably many nonzero projections $\{E_n\}$ in $\M$ such that $\sum_n E_n=I$ and $E_n XE_n=0$ for all $n$.
\end{cor}

\section{Type ${\rm II}_1$ case}

The following lemma is Proposition 3.1 of~\cite{KNZ}. We rewrite it for our purpose and provide the complete proof for readers' convenience.
\begin{lem}\label{L:lemma 4.1}
Let $(\M,\tau)$ be a finite factor, $A\in \M^+$ and let $N\in\NNN\cup\{\infty\}$. Then the following conditions are equivalent.
\begin{enumerate}
\item There is a decomposition of the identity into $N$ mutually orthogonal nonzero projections $E_j$, $I=\sum_{j=1}^NE_j$, for which $\sum_{j=1}^N E_jAE_j=I$, the convergence of the series being in the strong operator topology if $N=\infty$.
\item $\tau(A)=1$ and $A$ is the sum of $N$ nonzero projections, the convergence of the series being in the strong operator topology if $N=\infty$.
\end{enumerate}
\end{lem}
\begin{proof}
(1)$\Rightarrow$(2). $\sum_{j=1}^NE_jAE_j=I$ is equivalent to $E_jAE_j=E_j$ for each $1\leq j\leq N$. Let $P_j=A^{1/2}E_jA^{1/2}$ for $1\leq j\leq N$. Then $P_j$ is a projection for $1\leq j\leq N$ and $A=\sum_{j=1}^NP_j$. Note that $1=\tau(I)=\tau\left(\sum_{j=1}^N E_jAE_j\right)=\tau(A)$.

(2)$\Rightarrow$(1). Let $A=\sum_{j=1}^N P_j$ where $\{P_j\}$ are nonzero projections. Then $1=\tau(A)=\sum_{j=1}^N\tau(P_j)$. We decompose the identity $I=\sum_{j=1}^N F_j$ into $N$ mutually orthogonal projections $F_j\sim P_j$.  Now choose partial isometries $W_j$ with $P_j=W_j^*W_j$ and $F_j=W_jW_j^*$. If $N<\infty$, define $B=\sum_{j=1}^N W_j$. If $N=\infty$ and $n>m$, then
\[
\left(\sum_{j=m}^n W_j\right)^*\left(\sum_{j=m}^n W_j\right)=\sum_{i,j=m}^nW_i^*W_j=\sum_{j=m}^nW_j^*W_j=\sum_{j=m}^n P_j.
\]
Thus by the strong and hence the weak convergence of the series $\sum_{j=1}^\infty P_j$, we see that the series $\sum_{j=1}^\infty W_j$ is strongly Cauchy and hence converges in the strong operator topology. Call its sum $B$. Then $B^*B=\sum_{j=1}^NP_j=A$. Let $B=VA^{1/2}$ such that $V^*V=VV^*=I$. Then $BB^*=VAV^*$.  Moreover, $F_jB=W_j$ for every $j$, thus
\[
\sum_{j=1}^NF_jVAV^*F_j=\sum_{j=1}^NF_jBB^*F_j=\sum_{j=1}^NW_jW_j^*=\sum_{j=1}^NF_j=I.
\]
Let $E_j=V^*F_jV$. Then we have (1).
\end{proof}

\noindent{\bf Remark:}\, Note that in the above lemma, if ``$\tau(A)=1$" in (2) is changed to ``$\tau(A)=\tau(I)$", then (1)$\Rightarrow$(2) is true for arbitrary semifinite von Neumann algebras.

\begin{lem}\label{L:lemma 4.3}
Let $(\M,\tau)$ be a type ${\rm II}_1$ factor and let $A\in \M^+$ such that $\tau(A)=1$. Then $A$ is a strong sum of projections.
\end{lem}
\begin{proof}
Let $X=A-I$. Then $X=X^*$ and $\tau(X)=0$. By Corollary~\ref{C:Corollary 2.5}, there exists a sequences of mutually orthogonal nonzero projections
$\E=\{\{E_n\}_{n=1}^N\}$ such that $\sum_{n=1}^N E_n=I$ and $E_nXE_n=0$ for all $1\leq n\leq N$. Then $E_nAE_n=E_n$ for all $1\leq n\leq N$.
By Lemma~\ref{L:lemma 4.1}, we prove lemma~\ref{L:lemma 4.3}.
\end{proof}

\begin{thm}\label{T:theorem 2.7}
Let $\M$ be a type ${\rm II}_1$ factor and let $\tau$ be the faithful positive normal trace normalized by $\tau(I)=1$. Given $A\in \M^+$, we denote by $A_+=(A-I)\chi_A(1,\|A\|]$ the excess part of A and by $A_-=(I-A)\chi_A(0,1)$
the defect part of $A$. If $A\in \M^+$ and $\tau(A_+)\geq \tau(A_-)$, then $A$ can be written as the sum of a finite or infinite collection of projections.
\end{thm}
\begin{proof}
By considering $R_A\M R_A$, we may assume that $R_A=I$. Since $A=A_+-A_-+R_A$ and $\tau(A_+)\geq \tau(A_-)$, $\tau(A)\geq \tau(I)=1$. Write $s=\tau(A)$ and $B=\frac{A}{s}$. Then $B\geq 0$ and $\tau(B)=1$. By Lemma~\ref{L:lemma 4.3}, $B=P_1+P_2+\cdots+P_N$ and $A=sP_1+sP_2+\cdots+sP_N$. By Lemma 5.1 or Theorem 5.2 of~\cite{KNZ}, each $sP_k$ is a strong sum of projections in $\M$. Thus $A$ is a strong sum of projections.
\end{proof}

\section{Type ${\rm II}_\infty$ case}

\begin{lem}\label{L:lemma 3.1}
Let $(\M,\tau)$ be a type ${\rm II}_1$ factor and $A\in \M^+$. Let $\A$ be a separable diffuse abelian von Neumann subalgebra of $\M$ containing $A$. Then $\{\tau(AE):\,E \,\text{is a projection of}\, \A\}=[0,\tau(A)]$.
\end{lem}
\begin{proof}
$\{\tau(AE):\,E \,\text{is a projection of}\, \A\}\subseteq [0,\tau(A)]$ is clear. We need to show $\{\tau(AE):\,E \,\text{is a projection}\\\text{of}\, \A\}\supseteq[0,\tau(A)]$.
There is an isomorphism $\varphi$ from $(\A,\tau)$ onto $(L^\infty[0,1],m)$ such that $m\circ \varphi=\tau$. For $0\leq s\leq 1$, let $E_s=\varphi^{-1}(\chi_{[0,s]})$ and define
$f(s)=\tau(AE_s)$. Then $f(0)=0$ and $f(1)=\tau(A)$. Note that
\[
|f(s)-f(t)|=|\tau(A(E_s-E_t))|\leq \tau(|E_s-E_t|)\|A\|=\|A\||s-t|.
\]
This implies that $f(s)$ is a continuous function on $[0,1]$. Therefore, $\{\tau(AE):\,E \,\text{is a projection of}\, \A\}\supseteq[0,\tau(A)]$.
\end{proof}

\begin{lem}\label{L:lemma 3.2}
Let $(\M,\tau)$ be a type ${\rm II}_\infty$ factor and $A\in \M^+$ such that $\tau(A)<\infty$. Then there is an abelian von Neumann subalgebra $\A$  of $\M$ containing $A$ such that $\{\tau(AE):\,E \,\text{is a projection of}\, \A\, \text{with}\\ \tau(E)<\infty \,\text{and}\, E\leq R_A\}\supseteq [0,\tau(A))$.
\end{lem}
\begin{proof}
For every $n\in \NNN$, define $E_n=\chi_{A}[\frac{1}{n},+\infty)$. Then
\[
\frac{1}{n}\tau(E_n)\leq \int_{[\frac{1}{n}, +\infty)} \lambda d\tau(\chi_A[\lambda,+\infty))\leq \tau(A).
\]
Thus $\tau(E_n)\leq n\tau(A)<\infty$. Note that $E_1\leq E_2\leq E_3\leq \cdots$ and $\bigvee_n E_n=R_A$. Thus
\[
\tau(A)=\tau(AR_A)=\lim_n \tau(AE_n).
\]
Let $F_1=E_1$, $F_2=E_2-E_1$, $F_3=E_3-E_2$, $\cdots$. Then $\sum_{n=1}^\infty F_n=R_A$ and $F_n\M F_n$ is a type ${\rm II}_1$ factor. Choose a separable diffuse abelian von Neumann subalgebra  $\A_n$ of $F_n\M F_n$ containing $F_nAF_n$ and let $\A=\oplus \A_n\oplus \CCC(I-R_A)$. By Lemma~\ref{L:lemma 3.1},
\[\{\tau(AE):\,E \,\text{is a projection of}\, \A_n\}=[0,\tau(AF_n)].\] Now the lemma is clear.
\end{proof}

\begin{lem}\label{L:lemma 3.3}
Let $\M$ be a von Neumann algebra and $A\in \M^+$. If $E\in \M$ is a projection such that $EA=AE$, then $(AE)_+=A_+E$ and $(AE)_-=A_-E$.
\end{lem}
\begin{proof}
$(AE)_+=(AE-I)\chi_{AE}(1,+\infty)$ and $(AE)_-=(I-AE)\chi_{AE}(0,1)$. Claim for $s>0$, $\chi_{AE}[s,+\infty)=\chi_A[s,+\infty) E$. It is easy to see that there exists a sequence of decreasing continuous functions $f_n(t)$, $f_n(0)=0$, and $\chi_{[s,+\infty)}(t)=\lim f_n(t)$ pointwise. By the Lebesgue Dominated theorem, $\forall \xi,\eta\in \H$,
\[
\langle \chi_A{[s,+\infty)}\xi,\eta\rangle=\int \chi_{[s,+\infty)}(t)d\langle \chi_A[t,+\infty) \xi,\eta\rangle=\lim \int f_n(t)d\langle \chi_A[t,+\infty) \xi,\eta\rangle=\lim\langle f_n(A) \xi,\eta\rangle.
\]
Similarly, $\chi_{AE}[s,+\infty)=\lim f_n(AE)$ in the weak operator topology. Note that for a polynomial $p(t)$ with $p(0)=0$, we have $p(AE)=p(A)E$. By the Stone-Weierstrass theorem, $f_n(AE)=f_n(A)E$ and thus $\chi_{AE}[s,+\infty)=\chi_A[s,+\infty)E$.  Note that
\[
\chi_{AE}(1,+\infty)=\lim \chi_{AE}\left[1+\frac{1}{n},+\infty\right)=\lim \chi_{A}\left[1+\frac{1}{n},+\infty\right)E=\chi_{A}(1,+\infty)E
\]
and
\[
\chi_{AE}(0,1)=\lim \chi_{AE}\left[\frac{1}{n},+\infty\right)-\chi_{AE}[1,+\infty)=\lim \chi_{A}\left[\frac{1}{n},+\infty\right)E-\chi_{A}[1,+\infty)E=\chi_{A}(0,1)E.
\]
Thus $(AE)_+=A_+E$ and $(AE)_-=A_-E$.
\end{proof}

\begin{thm}\label{T:theorem 4.4}
Let $(\M,\tau)$ be a type ${\rm II}_\infty$ factor and $A\in \M^+$. If $\tau(A_+)\geq \tau(A_-)$, then $A$ is a strong sum of projections.
\end{thm}
\begin{proof}
By Theorem 6.6 of~\cite{KNZ}, if $\tau(A_+)=\infty$, then $A$ is a strong sum of projections. So we may assume that $\tau(A_-)\leq \tau(A_+)<\infty$. We divide the proof into two cases.

Case 1. $\tau(A_+)=\tau(A_-)<\infty$. If $\tau(A_+)=\tau(A_-)=0$, then $A$ is a projection. So we may assume that $\tau(A_+)=\tau(A_-)>0$. By Lemma~\ref{L:lemma 3.2}, there are finite projections $e_1$ and $f_1$ such that $e_1A_+=A_+e_1$, $f_1A_-=A_-f_1$, $e_1\leq R_{A_+}$, $f_1\leq R_{A_-}$, and $\tau(e_1A_+)=\tau(f_1A_-)=\frac{1}{2}\tau(A_+)$. Since $A_+=(A-I)\chi_{A}(1,\|A\|]$ and $A_-=(I-A)\chi_A(0,1)$, $R_{A_+}=\chi_{A}(1,\|A\|]$ and $R_{A_-}=\chi_A(0,1)$. Thus
$e_1\leq \chi_{A}(1,\|A\|]$ and $f_1\leq \chi_A(0,1)$.
Then
\[
e_1A=e_1 (A-I)\chi_{A}(1,\|A\|]+e_1\chi_{A}(1,\|A\|]+e_1A\chi_{A}[0,1]\]\[=(A-I)\chi_{A}(1,\|A\|]e_1+\chi_{A}(1,\|A\|]e_1+A\chi_{A}[0,1]e_1=Ae_1
\]
and
\[
f_1A=f_1 (A-I)\chi_{A}(0,1)+f_1\chi_A(0,1)+f_1\chi_A{\{0\}}A+f_1\chi_A[1,\|A\|]A
\]
\[
=(A-I)\chi_{A}(0,1)f_1+\chi_A(0,1)f_1+\chi_A{\{0\}}Af_1+\chi_A[1,\|A\|]Af_1=Af_1.
\]
 Let $E_1=e_1+f_1$. Then $AE_1=E_1A$. By Lemma~\ref{L:lemma 3.3},
\[
(AE_1)_+=A_+E_1=A_+e_1
\]
and
\[
(AE_1)_-=A_-E_1=A_-f_1.
\]
Thus $\tau((AE_1)_+)=\tau((AE_1)_-)$. Note that $AE_1\in E_1\M E_1$ and $E_1\M E_1$ is type ${\rm II}_1$. By Theorem~\ref{T:theorem 2.7}, $AE_1$ is a strong sum of projections. Note that
\[
(A(I-E_1))_+=A_+(I-E_1)=A_+-A_+e_1
\]
and
\[
(A(I-E_1))_-=A_-(I-E_1)=A_- -A_-f_1.
\] Then $\tau((A(I-E_1))_+)=\tau((A(I-E_1))_-)=\frac{1}{2}\tau(A_+)<+\infty$. Let $A_1=A$ and $A_2=A(I-E_1)$. Then $A_2\leq A_1$. By Lemma~\ref{L:lemma 3.2}, there are finite projections $e_2$ and $f_2$ such that $e_2(A_2)_+=(A_2)_+e_2$, $f_2(A_2)_-=(A_2)_-f_2$, $e_2\leq R_{(A_2)_+}$, $f_2\leq R_{(A_2)_-}$, and $\tau(e_2(A_2)_+)=\tau(f_2(A_2)_-)=\frac{1}{2}\tau((A_2)_+)$. Since $(A_2)_+=(A_2-I)\chi_{A_2}(1,\|A_2\|]$ and $(A_2)_-=(I-A_2)\chi_{A_2}(0,1)$, $R_{(A_2)_+}=\chi_{A_2}(1,\|A_2\|]$ and $R_{(A_2)_-}=\chi_{A_2}(0,1)$. Thus
$e_2\leq \chi_{A_2}(1,\|A_2\|]\leq R_{(A_2)}\leq I-E_1$ and  $f_2\leq \chi_{A_2}(0,1)\leq R_{(A_2)}\leq I-E_1$.
Then
\[
e_2A=e_2A_2=e_2 (A_2-I)\chi_{A_2}(1,\|A_2\|]+e_2\chi_{A_2}(1,\|A_2\|]+e_2A_2\chi_{A_2}[0,1]\]\[=(A_2-I)\chi_{A_2}(1,\|A_2\|]e_2+\chi_{A_2}(1,\|A_2\|]e_2+A_2\chi_{A_2}[0,1]e_2=A_2e_2=Ae_2
\]
and
\[
f_2A=f_2A_2=f_2 (A_2-I)\chi_{A_2}(0,1)+f_2\chi_{A_2}(0,1)+f_2\chi_{A_2}{\{0\}}A_2+f_2\chi_{A_2}[1,\|A_2\|]A_2
\]
\[
=(A_2-I)\chi_{A_2}(0,1)f_2+\chi_{A_2}(0,1)f_2+\chi_{A_2}{\{0\}}A_2f_2+\chi_{A_2}[1,\|A_2\|]A_2f_2=A_2f_2=Af_2.
\]
 Let $E_2=e_2+f_2$. Then $AE_2=A_2E_2=E_2A_2=E_2A$. By Lemma~\ref{L:lemma 3.3},
\[
(AE_2)_+=(A_2E_2)_+=(A_2)_+E_2=(A_2)_+e_2
\]
and
\[
(AE_2)_-=(A_2E_2)_-=(A_2)_-E_2=(A_2)_-f_2.
\]
 Thus
 \[\tau((AE_2)_+)=\tau\left((A_2)_+e_2\right)=\frac{1}{2}\tau((A_2)_+)=\frac{1}{4}\tau(A_+)\]
 and
 \[\tau((AE_2)_-)=\tau\left((A_2)_-f_2\right)=\frac{1}{2}\tau((A_2)_-)=\frac{1}{4}\tau(A_-).\]
 Therefore, $\tau((AE_2)_+)=\tau((AE_2)_-)$. Note that $AE_2\in E_2\M E_2$ and $E_2\M E_2$ is type ${\rm II}_1$. By Theorem~\ref{T:theorem 2.7}, $AE_2$ is a strong sum of projections. By Lemma~\ref{L:lemma 3.3},
\[
(A(I-E_1-E_2))_+=
A_+(I-E_1-E_2)=A_+-A_+e_1-(A_2)_+e_2
\]
and
\[
(A(I-E_1-E_2))_-=A_-(I-E_1-E_2)=A_- -A_-f_1-(A_2)_-f_2.
\] Then
\[\tau((A(I-E_1-E_2))_+)=\tau((A(I-E_1-E_2))_-)=\frac{1}{4}\tau(A_+)<+\infty.\]
  Let $A_3=A(I-E_1-E_2)$. Then $A_3\leq A_2\leq A_1$.
By a similar argument, there is a finite projection $E_3\in \M$ such that $E_3\leq R_{A_3}\leq I-E_1-E_2$, $E_3A_3=E_3A=AE_3=A_3E_3$, and
 \[\tau((A_3E_3)_+)=\tau((A_3E_3)_-)=\frac{1}{8}\tau(A_+)<+\infty.\]
  By Lemma~\ref{L:lemma 3.3},
\[
(A(I-E_1-E_2-E_3))_+=
A_+(I-E_1-E_2-E_3)=A_+-A_+e_1-(A_2)_+e_2-(A_3)_+e_3
\]
and
\[
(A(I-E_1-E_2-E_3))_-=A_-(I-E_1-E_2-E_3)=A_- -A_-f_1-(A_2)_-f_2-(A_3)_-f_3.
\] Then
\[\tau((A(I-E_1-E_2-E_3))_+)=\tau((A(I-E_1-E_2-E_3))_-)=\frac{1}{8}\tau(A_+)<+\infty.\]

 By induction, there exist $A=A_1\geq A_2\geq A_3\geq\cdots$ and finite projections $E_1,E_2,E_3,\cdots$ of $\M$ such that $E_n\leq I-\sum_{k=1}^{n-1} E_k$, $AE_n=A_nE_n=E_nA_n=E_nA$, $A_n=A(I-\sum_{k=1}^{n-1}E_k)$ and
 \[
 \tau((AE_n)_+)=\tau((A_nE_n)_+)=\tau((AE_n)_-)=\tau((A_nE_n)_-)=\frac{1}{2^n}\tau(A_+)<+\infty.
 \]

 Let $B=A(I-\sum_{n=1}^\infty E_n)=\lim_n A_n$ in the strong operator topology. Then by Lemma~\ref{L:lemma 3.3},
 \[
 \tau(B_+)=\tau(A_+(I-\sum_{n=1}^\infty E_n))=0
 \]
 and
 \[
 \tau(B_-)=\tau(A_-(I-\sum_{n=1}^\infty E_n))=0.
 \]
 This implies that $B$ is a projection. By Theorem~\ref{T:theorem 2.7}, $AE_n\in E_n\M E_n$ is a strong sum of projections for every $n$. Thus
 \[
 A=\sum_{n=1}^\infty AE_n+B
 \]
  is a strong sum of projections.

 Case 2. Assume $\tau(A_-)<\tau(A_+)<\infty$. If $\tau(A_-)>0$, then $s=\tau(A_+)-\tau(A_-)<\tau(A_+)$. By Lemma~\ref{L:lemma 3.2}, there is a finite projection $E\in \M$ such that $E\leq R_{A_+}$, $EA=AE$ and $\tau(A_+E)=s$. Let $F=I-E$. Since $R_{A_+}=\chi_{A}(1,+\infty)$ and $R_{A_-}=\chi_A(0,1)$, $E$ is orthogonal to $R_{A_-}$. By Lemma~\ref{L:lemma 3.3},
 \[
 \tau((AE)_+)=\tau(A_+E)=s>0,
 \]
 \[
 \tau((AE)_-)=\tau(A_-E)=0,
 \]
 \[
 \tau((AF)_+)=\tau(A_+F)=\tau(A_+)-\tau(A_+E)=\tau(A_+)-s=\tau(A_-)\]
 and
 \[ \tau((AF)_-)=\tau(A_-F)=\tau(A_-).
 \]
 By Case 1, $AF$ is a strong sum of projections. Note that $EAE\in E\M E$ and $E\M E$ is a type ${\rm II}_1$ factor.  By Theorem~\ref{T:theorem 2.7}, $EAE$ is also a strong sum of projections. Thus $A$ is a strong sum of projections.

 Assume $\tau(A_-)=0$. Then by Lemma~\ref{L:lemma 3.2}, there is a sequence of finite projections $F_n$, $\sum F_n=R_{A_+}$ and $\tau(A_+ F_n)=\frac{1}{2^n}\tau(A_+)$. Now
 $A=\sum_n AF_n+A(I-R_{A_+})$. By Lemma~\ref{L:lemma 3.3},
 \[
 \tau((AF_n)_+)=\tau(A_+F_n)=\frac{1}{2^n}\tau(A_+)
 \]
 and
 \[
 \tau((AF_n)_-)=\tau(A_-F_n)=0.
 \]
  By Theorem~\ref{T:theorem 2.7}, $AF_n\in F_n\M F_n$ is a strong sum of projections. Note that
 \[
 \tau((A(I-R_{A_+}))_+)=\tau(A_+)-\sum_n \frac{1}{2^n}\tau(A_+)=0
 \]
 and
 \[
 \tau((A(I-R_{A_+}))_-)=\tau(A_-(I-R_{A_+}))=0.
 \]
 This implies that $A(I-R_{A_+})$ is a projection. Thus $A$ is a strong sum of projections.
\end{proof}

\noindent{\bf Remark:}\, If in Theorem~\ref{T:theorem 4.4} we assume that $\M$ is countably decomposable, then we have a simple proof of Case 1 as follows. By restricted on $R_A\M R_A$, we may assume that $R_A=I$. Then $A=A_+-A_-+I$. Let $X=A_+-A_-$. Then $A=X+I$. Since $\tau(A_+)=\tau(A_-)<\infty$, $X\in L^1(\M,\tau)$ and $\tau(X)=0$. By Theorem~\ref{T:theorem 2.4}, there is a sequence of projections $\{E_n\}_{n=1}^N$ such that $\sum_{n=1}^N E_n=I$ and $E_nXE_n=0$ for all $1\leq n\leq N$. Then $E_nAE_n=E_n$ for all $1\leq n\leq N$. By Lemma~\ref{L:lemma 4.1}, $A$ is a strong sum of projections.

\section{A new proof of type ${\rm I}$ case}

\begin{thm}
Let $(\M,\tau)$ be a type ${\rm I}_n$ factor with the trace $\tau$ such that $\tau(I)=n$. If $A\in \M^+$ satisfies $\tau(A)\in \NNN\cup\{0\}$ and $\tau(A)\geq \tau(R_A)$, then $A$ is the sum of projections.
\end{thm}
\begin{proof}
By restricted on $R_A\M R_A$, we may assume that $R_A=I$. Let $\lambda_1\geq \lambda_2\geq\cdots\geq \lambda_n$ be eigenvalues of $A$. If $\tau(A)>\tau(I)=n$, then $\lambda_1>1$. Let $e$ be a unit eigenvector corresponding to $\lambda_1$, then $A-e\otimes e\in \M^+$, $\tau(A-e\otimes e)\in \NNN$, and $\tau(A-e\otimes e)\geq n$. By induction, we may assume that $\tau(A)=\tau(I)=n$. Let $X=A-I$. Then $X=X^*$ and $\tau(X)=0$. By Theorem~\ref{T:theorem 2.4},
there are finite nonzero projections $E_1+\cdots+E_N=I$ such that $E_iXE_i=0$ for $1\leq i\leq N$. Then
 $E_iAE_i=E_i$ for $1\leq i\leq N$. By Lemma~\ref{L:lemma 4.1}, $A$ is the sum of $N$ projections.
\end{proof}

\begin{thm}
Let $\H$ be the infinite dimensional separable Hilbert space and let $\M=\B(\H)$ with the classical tracial weight $\rm{Tr}$. Suppose $A\in \M^+$. If either ${\rm Tr}(A_+)=\infty$ or ${\rm Tr}(A_-)\leq {\rm Tr}(A_+)<\infty$ and ${\rm Tr}(A_+)-{\rm Tr}(A_-)\in \NNN\cup \{0\}$, then $A$ is a strong sum of projections.
\end{thm}
\begin{proof}
The case ${\rm Tr}(A_+)=\infty$ follows from Theorem 6.6 of~\cite{KNZ}. In the following, we assume that ${\rm Tr}(A_-)\leq {\rm Tr}(A_+)<\infty$ and ${\rm Tr}(A_+)-{\rm Tr}(A_-)\in \NNN\cup \{0\}$. By restricted on $R_A\M R_A$, we may assume that $R_A=I$. Since $A_+$ and $A_-$ are trace-class operators and supported on orthogonal subspaces, they are simultaneously diagonalizable, so we can set $A_-=\sum_{i=1}^M\lambda_i(f_i\otimes f_i)$, $A_+=\sum_{j=1}^N \mu_j (e_j\otimes e_j)$, where $M,N\in\NNN\cup\{0\}\cup\{\infty\}$, $\{f_i,e_j\}$ are mutually orthogonal unit vectors, and $0<\lambda_i<1$, $\mu_j>0$ for all $i$ and $j$. Since $A=A_+-A_-+I$, $A$ is a diagonalizable operator. If ${\rm Tr}(A_+)-{\rm Tr}(A_-)>0$, then we have $N\geq 1$ and $A=(A-e_1\otimes e_1)+e_1\otimes e_1$. Note that $R_{A-e_1\otimes e_1}=I$ and ${\rm Tr}((A-e_1\otimes e_1)_+)-{\rm Tr}((A-e_1\otimes e_1)_-)\geq 0$. By induction, we may assume that $R_A=I$ and  ${\rm Tr}(A_+)-{\rm Tr}(A_-)=0$. Let $X=A_+-A_-$. Then ${\rm Tr}(X)=0$. By Theorem~\ref{T:theorem 2.4}, there are mutually orthogonal nonzero projections $E_1,\cdots,E_N$ such that $E_1+\cdots+E_N=I$ and $E_kXE_k=0$ for each $1\leq k\leq N$. Now we have $E_kAE_k=E_k(I+X)E_k=E_k$ for each $1\leq k\leq N$. Thus by Lemma~\ref{L:lemma 4.1}, $A$ is a sum of $N$ projections.
\end{proof}

\end{document}